\documentclass[10pt]{amsart}
\usepackage{enumerate,amsmath,amssymb,latexsym,
amsfonts, amsthm, amscd, MnSymbol}

\theoremstyle{plain}

\newtheorem{theorem}{Theorem}

\numberwithin{equation}{section}

\newcommand{\ii}{{\rm i}}

\begin{document}

\title {Elliptic functions from $F(\frac{1}{4}, \frac{3}{4} ; \frac{1}{2} ; \bullet)$ }

\date{}

\author[P.L. Robinson]{P.L. Robinson}

\address{Department of Mathematics \\ University of Florida \\ Gainesville FL 32611  USA }

\email[]{paulr@ufl.edu}

 \keywords{}

\begin{abstract}
We reconsider the elliptic functions that are generated from the hypergeometric function $F(\tfrac{1}{4}, \tfrac{3}{4}; \tfrac{1}{2} ; \bullet)$ by Li-Chien Shen, presenting fresh proofs that do not require the use of theta functions. 
\end{abstract}

\maketitle

\medbreak

\section*{Introduction}

\medbreak 

Li-Chien Shen has made interesting contributions to the Ramanujan theory of elliptic functions to alternative bases. Among these contributions are [1] and [2]: in [1] he developed a family of elliptic functions from the hypergeometric function $F(1/3, 2/3; 1/2 ; \bullet)$; in [2]  he developed a family of elliptic functions from the hypergeometric function $F(1/4, 3/4; 1/2 ; \bullet)$. In both papers, the analysis includes an expression for each elliptic function in terms of its coperiodic Weierstrass $\wp$-function, along with much other information; in both papers, the derivation depends essentially on the theory of the corresponding theta functions, about which also much information is presented. Our purpose here is to address the material of [2] as it relates to the elliptic functions and their corresponding $\wp$-functions directly, without the theta functions as intermediaries; it is a sequel to [0], which performed a like service for [1]. In an appendix, we trace the origins of the elliptic functions in [1] and [2] from the classical theory. For further detail, though with slightly different notation, we refer to the original papers [1] and [2]. 

\medbreak 

\section*{The elliptic functions} 

\medbreak 

Let $0 < \kappa < 1$ and $\lambda:= \sqrt{1 - \kappa^2}$. Write 
$$u = \int_0^{\phi} F(\tfrac{1}{4}, \tfrac{3}{4}; \tfrac{1}{2} ; \kappa^2 \sin^2 \theta) \, {\rm d} \theta = \int_0^{\sin \phi} F(\tfrac{1}{4}, \tfrac{3}{4}; \tfrac{1}{2} ; \kappa^2 t^2) \frac{{\rm d} t}{\sqrt{1 - t^2}}.$$
Near the origin, the relation $\phi \mapsto u(\phi)$ inverts to $u \mapsto \phi(u)$ with $0$ as a fixed point; moreover, there exists $\psi$ fixing $0$ and satisfying 
$$\sin \psi = \kappa \sin \phi.$$ 
Now, we define functions $s, c$ and $d$ by 
$$s = \sin \phi$$ 
$$c = \cos \phi$$
and  
$$d = \cos \psi.$$ 
Plainly, these functions satisfy the `Jacobian' identities 
$$s^2 + c^2 = 1 \; \; \; {\rm and} \; \; \; \kappa^2 s^2 + d^2 = 1.$$ 

\medbreak 

Our immediate aim is to prove that $d$ is an elliptic function and to identify it as a function of its coperiodic Weierstrass function. Of course, here we more strictly mean to say that $d$ extends to an elliptic function; similarly for other functions that arise later. 

\medbreak 

\begin{theorem} \label{phipsi}
The functions $\phi$ and $\psi$ have derivatives given by 
$$\phi' = \frac{\cos \psi}{\cos \frac{1}{2} \psi} \; \; \; {\rm and} \;\;\; \psi' = \kappa \, \frac{\cos \phi}{\cos \frac{1}{2} \psi}.$$
\end{theorem} 

\begin{proof} 
From the definition of $u$ it follows that $u' = F(\tfrac{1}{4}, \tfrac{3}{4}; \tfrac{1}{2} ; \kappa^2 \sin^2 \phi) = F(\tfrac{1}{4}, \tfrac{3}{4}; \tfrac{1}{2} ; \sin^2 \psi)$
whence by inversion and the hypergeometric identity 
$$F(\tfrac{1}{4}, \tfrac{3}{4}; \tfrac{1}{2} ; \sin^2 \psi) = \frac{\cos \frac{1}{2} \psi}{\cos \psi}$$
we deduce the first formula:  
$$\phi' = \frac{\cos \psi}{\cos \frac{1}{2} \psi}.$$
From the definition of $\psi$ it follows by differentiation that 
$$(\cos \psi) \, \psi' = \kappa (\cos \phi) \, \phi'$$
whence substitution of the first formula yields the second: 
$$\psi' = \kappa \, \frac{\cos \phi}{\cos \frac{1}{2} \psi}.$$
\end{proof} 

\medbreak 

The derivative of the function $d$ may now be presented in similar terms.

\medbreak 

\begin{theorem} \label{d'}
The derivative of $d$ is given by 
$$d\,' = - 2 \kappa \sin \tfrac{1}{2} \psi \cos \phi.$$
\end{theorem} 

\begin{proof} 
Differentiate using Theorem \ref{phipsi} and trigonometric duplication: thus, 
$$d\,' = (\cos \psi)' = - (\sin \psi) \, \psi' = - (\sin \psi) \, \kappa \, \frac{\cos \phi}{\cos \frac{1}{2} \psi} = - 2 \kappa \sin \tfrac{1}{2} \psi \cos \phi.$$
\end{proof} 

\medbreak 

We are now in a position to derive a first-order differential equation satisfied by $d$. 

\medbreak 

\begin{theorem} \label{dequation}
The function $d$ satisfies the differential equation 
$$(d\,')^2 = 2 \, (1 - d) \, (d^2 - \lambda^2).$$ 
\end{theorem} 

\begin{proof} 
Square the result of Theorem \ref{d'}: on the one hand, 
$$4 \sin^2 \tfrac{1}{2} \psi = 2 (1 - \cos \psi) = 2 (1 - d)$$
by trigonometric duplication; on the other hand, 
$$\kappa^2 \, \cos^2 \phi = \kappa^2 - \kappa^2 \, \sin^2 \phi = \kappa^2 - \sin^2 \psi = \kappa^2 - 1 + \cos^2 \psi = - \lambda^2 + d^2$$
by definition of $\psi$ and $\lambda$. 
\end{proof} 

\medbreak 

Up until this point, our presentation closely follows that of [2]; from this point on, we begin to deviate. 

\medbreak 

The differential equation of Theorem \ref{dequation} enables us to see that $d$ is elliptic; indeed, it enables us to identify $d$ in terms of its coperiodic $\wp$-function. 

\medbreak 

\begin{theorem} \label{wp}
The function $d$ is elliptic, being given by 
$$d = 1 - \frac{\tfrac{1}{2} \kappa^2}{\wp + \tfrac{1}{3}}$$ 
where $\wp = \wp(\bullet; g_2, g_3)$ is the Weierstrass function with invariants 
$$g_2 = \tfrac{4}{3} - \kappa^2 \; \; \; {\rm and} \; \; \; g_3 = \tfrac{8}{27} - \tfrac{1}{3} \kappa^2.$$
\end{theorem} 

\begin{proof} 
 In [2] this is proved by reference to a standard formula for the integral of $f^{-1/2}$ when $f$ is a quartic. Here, we proceed directly from the differential equation. First, the substitution $r = (1 - d)^{-1}$ yields 
$$(r\,')^2 = 2 \kappa^2 r^3 - 4 r^2 + 2 r;$$
next, the substitution $q = \tfrac{1}{2} \kappa^2 r$ leads to 
$$(q\,')^2  = 4 q^3 - 4 q^2 + \kappa^2 q;$$ 
lastly, the substitution $p = q - \tfrac{1}{3}$ produces 
$$(p\,')^2 = 4 p^3 - (\tfrac{4}{3} - \kappa^2) p - (\tfrac{8}{27} - \tfrac{1}{3} \kappa^2).$$
At the origin, $d = 1$ whence $p$ has a pole; so $p$ is $\wp$ with the stated invariants. 
\end{proof} 

\medbreak 

The formula of Theorem \ref{wp} is perhaps better expressed thus: 
$$(d - 1) \, (\wp + \tfrac{1}{3}) = - \tfrac{1}{2} \kappa^2.$$ 

\bigbreak 

Notice that the differential equation of Theorem \ref{dequation} is cubic in $d$; this circumstance suggests an alternative identification of $d$. First, `reverse signs': substitute $r = -d$ to obtain 
$$(r\,')^2 = 2 r^3 + 2 r^2 - 2 \lambda^2 r - 2 \lambda^2.$$
Next, rescale: substitute $q = \tfrac{1}{2} r$ and deduce that 
$$(q\,')^2 = 4 q^3 + 2 q^2 - \lambda^2 q - \tfrac{1}{2} \lambda^2.$$
Lastly, translate: substitute $p = q + \tfrac{1}{6}$ and conclude that 
$$(p\,')^2 = 4 p^3 - (\lambda^2 + \tfrac{1}{3}) p - (\tfrac{1}{3} \lambda^2 - \tfrac{1}{27}).$$
Here, $\lambda^2 + \tfrac{1}{3} = \tfrac{4}{3} - \kappa^2$ and $\tfrac{1}{3} \lambda^2 - \tfrac{1}{27} = \frac{8}{27} - \frac{1}{3} \kappa^2$; thus, the `new' $p$ and the `old' $p$ (in the proof of Theorem \ref{wp}) satisfy the same differential equation. However, the fact that $d = 1$ at the origin gives this `new' $p$ the value $-\tfrac{1}{3}$ there. Accordingly, this $p$ is a translate of $\wp$: in fact, $p(u) = \wp(u + a)$ where $a$ is such that $\wp (a) = -\tfrac{1}{3}$. Unravelling the substitutions, we arrive at the alternative identification 
$$d(u) = \tfrac{1}{3} - 2 \, \wp(u + a).$$
A glance ahead at Theorem \ref{mid} shows that we may take $a$ to be $\ii \, \omega'$: thus  
$$d(u) = \tfrac{1}{3} - 2 \, \wp(u + \ii \, \omega').$$ 

\bigbreak 

The Weierstrass function $\wp$ has real invariants and has discriminant 
$$g_2^3 - 27 g_3^2 = \kappa^4 \, (1 - \kappa^2) > 0$$
whence its period lattice is rectangular; take its fundamental periods to be $2 \, \omega$ and $2 \, \ii \, \omega'$ with $\omega > 0$ and $\omega' > 0$. Its midpoint values are the zeros of the cubic 
$$4 z^3 - g_2 z - g_3 = 4 z^3 - (\tfrac{4}{3} - \kappa^2) z - (\tfrac{8}{27} - \tfrac{1}{3} \kappa^2) = 4 z^3 - (\lambda^2 + \tfrac{1}{3}) z - (\tfrac{1}{3} \lambda^2 - \tfrac{1}{27});$$
these zeros are $- \frac{1}{3}$ and $\frac{1}{6} \pm \frac{1}{2} \lambda$ precisely.

\medbreak 

\begin{theorem} \label{mid}
The midpoint values of $\wp$ are given by $$\wp(\ii \, \omega') = - \tfrac{1}{3}, \; \; \wp(\omega + \ii \, \omega') = \tfrac{1}{6} - \tfrac{1}{2} \lambda \; \; {\rm and} \; \; \wp(\omega)= \tfrac{1}{6} + \tfrac{1}{2} \lambda.$$
\end{theorem} 

\begin{proof} 
Supplement the information regarding the zeros of the cubic with the fact that $\wp(u)$ is real and strictly increases as $u$ traverses the rectangular path $0 \to \ii \, \omega' \to \omega + \ii \, \omega' \to \omega \to 0$.
\end{proof} 

\medbreak 

Each of these midpoint values of $\wp$ is double, of course: in particular, the function $\wp + \frac{1}{3}$ has a double zero at $\ii \, \omega'$. 

\medbreak 

Now, just as $\wp$ is an elliptic function of order two, with $2 \, \omega$ and $2 \, \ii \, \omega$ as fundamental periods, so also is $d$. Its poles are double, located at $\ii \, \omega'$ and points congruent modulo the period lattice, as follows. 

\medbreak 

\begin{theorem} \label{dvalues}
The elliptic function $d$ has a double pole at the point $\ii \, \omega'$; moreover, the (double) values of $d$ at $\omega + \ii \, \omega'$ and $\omega$ are $ - \lambda$ and $\lambda$ respectively.
\end{theorem} 

\begin{proof} 
The function $d$ has poles precisely where $\wp$ has its midpoint value $- \frac{1}{3}$ (according to Theorem \ref{wp}) hence precisely at $\ii \, \omega'$ and points congruent (according to Theorem \ref{mid}). As noted after Theorem \ref{mid}, the zero of $\wp + \frac{1}{3}$ at $\ii \, \omega'$ is double; by Theorem \ref{wp}, it follows that the pole of $d$ at $\ii \, \omega'$ is double likewise. Similar direct substitution from Theorem \ref{mid} into Theorem \ref{wp} completes the proof. 
\end{proof} 

\medbreak 

Just as the Weierstrass function $\wp$ is `real', so is the elliptic function $d$: that is, $\wp(\overline{u}) = \overline{\wp(u)}$ and so $d(\overline{u}) = \overline{d(u)}$. 

\medbreak 

\begin{theorem} \label{dzero}
The elliptic function $d$ has a simple zero $z^+$ on the interval $(\omega, \omega + \ii \, \omega')$ and a simple zero $z^- = \overline{z^+}$ on the interval $(\omega, \omega - \ii \, \omega')$. 
\end{theorem} 

\begin{proof} 
As $d$ is real-valued on the interval $(\omega, \omega + \ii \, \omega')$, Theorem \ref{dvalues} and the intermediate value theorem place a zero of $d$ on this interval. By `reality', the conjugate of this zero is also a zero. The simplicity of these zeros follows from the fact that zeros of $d$ are precisely zeros of $\wp + \tfrac{1}{3} - \tfrac{1}{2} \kappa^2$ and the fact that $(\wp\,')^2 = \tfrac{1}{2} \kappa^4 (\kappa^2 - 1) \neq 0$ at these points. 
\end{proof} 

\medbreak 

The simplicity of the zeros in Theorem \ref{dzero} is otherwise plain: their difference is not a period and the elliptic function $d$ has order two. Of course, it is likewise plain that the points congruent either to $z^+$ or to $z^-$ account for all the zeros of $d$. 

\medbreak 

Associated to the Weierstrass function $\wp$, with the real $2 \, \omega$ and the imaginary $2 \, \ii \, \omega'$ as fundamental periods, there is a corresponding triple ${\rm sn}, \, {\rm cn}, \, {\rm dn}$ of classical Jacobian elliptic functions. Recall that if $e_1 > e_3 > e_2$ are the midpoint values of a (more general) Weierstrass function then the corresponding Jacobian functions have modulus $k$ such that 
$$k^2 = \frac{e_3 - e_2}{e_1 - e_2}$$
and ${\rm sn}$ is related to $\wp$ by the formula
$$\wp(u) = e_2 + \frac{e_1 - e_2}{{\rm sn}^2[(e_1 - e_2)^{\tfrac{1}{2}} \, u]}\, .$$
In the present case, it follows from Theorem \ref{mid} that ${\rm sn}, \, {\rm cn}, \, {\rm dn}$ have modulus $k$ such that  
$$k^2 = \frac{\tfrac{1}{6} - \tfrac{1}{2} \lambda + \tfrac{1}{3}}{\tfrac{1}{6} + \tfrac{1}{2} \lambda + \tfrac{1}{3}} = \frac{1 - \lambda}{1 + \lambda};$$ 
equivalently 
$$\lambda = \frac{1 - k^2}{1 + k^2}$$
and therefore 
$$\kappa^2 = \frac{4 k^2}{(1 + k^2)^2}.$$ 
It also follows from Theorem \ref{mid} that here, ${\rm sn}$ is related to $\wp$ by
$$\wp(u) + \tfrac{1}{3} = \frac{\tfrac{1}{2} (1 + \lambda)}{{\rm sn}^2[(\tfrac{1}{2} (1 + \lambda))^{\tfrac{1}{2}} \, u]}\, .$$

\bigbreak 

We may now express the elliptic function $d$ explicitly in Jacobian terms. 

\medbreak 

\begin{theorem} \label{Jac} 
The elliptic function $d$ has Jacobian form  
$$d(u) = 1 - (1 - \lambda) \, {\rm sn}^2[(\tfrac{1}{2} (1 + \lambda))^{\tfrac{1}{2}} \, u].$$
\end{theorem} 

\begin{proof} 
Simply compare the formula displayed prior to the present theorem with the formula displayed in Theorem \ref{wp}. 
\end{proof} 

\medbreak 

Equivalently, since 
$$\tfrac{1}{2} (1 + \lambda) = \frac{1}{1 + k^2} \; \; {\rm and} \; \; 1 - \lambda = \frac{2 k^2}{1 + k^2}$$
we may write 
$$d(u) = 1 - \frac{2 k^2}{1 + k^2} \, {\rm sn}^2[(1 + k^2)^{-\tfrac{1}{2}} \, u].$$

\bigbreak 

The remaining `mock' Jacobian functions $s$ and $c$ are also identifiable in terms of the `true' Jacobian functions ${\rm sn}, \, {\rm cn}, \, {\rm dn}$. 

\medbreak 

\begin{theorem} \label{s}
The function $s$ satisfies 
$$s^2(u) = {\rm sn}^2 [(1 + k^2)^{-\tfrac{1}{2}} \, u] \Big\{k^2 + {\rm dn}^2 [(1 + k^2)^{-\tfrac{1}{2}} \, u] \Big\}.$$
\end{theorem} 

\begin{proof} 
Recall that $\kappa^2 s^2 = 1 - d^2 = (1 - d) \, (1 + d)$. Here, writing $v = (1 + k^2)^{-\tfrac{1}{2}} \, u$ for brevity, Theorem \ref{Jac} yields 
$$1 - d(u) = \frac{2 k^2}{1 + k^2} \, {\rm sn}^2 (v)$$
and 
$$1 + d(u) = 2 \, \big(1 - \frac{k^2}{1 + k^2}\, {\rm sn}^2 (v)\big) = \frac{2}{1 + k^2} \, \big(k^2 + 1 - k^2 {\rm sn}^2 (v)\big) = \frac{2}{1 + k^2} \, \big(k^2 + {\rm dn}^2 (v)\big)$$
on account of the (true) Jacobian identity $k^2 {\rm sn}^2 + {\rm dn}^2 = 1$. All that remains is to assemble the pieces and cancel the coefficient $\kappa^2 = \frac{4 k^2}{(1 + k^2)^2}.$
\end{proof} 

\medbreak 

It is perhaps worth remarking here that although $s^2$ is thus elliptic, it does not have a meromorphic square-root; in particular, $s$ itself is not an elliptic function. To see this, we inspect the zeros of $s^2$. The identity $\kappa^2 s^2 = (1 - d) \, (1 + d)$ shows that $s = 0$ precisely where $d = \pm 1$. On the one hand, $d$ has value $+ 1$ precisely where $\wp$ has poles; these points are double, both as poles of $\wp$ and as zeros of $d - 1$. On the other hand, $d = -1$ precisely where $\wp = \tfrac{1}{4} \kappa^2 - \tfrac{1}{3}$; there $(\wp\,')^2 = \tfrac{1}{16} k^6 \neq 0$, whence these points are simple. Thus $s^2$ has a simple zero and so lacks meromorphic square-roots. 

\medbreak 

\begin{theorem} \label{c}
The function $c$ satisfies 
$$c^2 (u) = {\rm cn}^2 [(1 + k^2)^{-\tfrac{1}{2}} \, u] \, {\rm dn}^2 [(1 + k^2)^{-\tfrac{1}{2}} \, u].$$ 
\end{theorem} 

\begin{proof} 
Again write $v = (1 + k^2)^{-\tfrac{1}{2}} \, u$ for convenience. Recall that $c^2 = 1 - s^2$ and compute: Theorem \ref{s} yields 
$$c^2 (u) = 1 - k^2 \, {\rm sn}^2(v)  - {\rm sn}^2 (v) \, {\rm dn}^2 (v) = {\rm dn}^2 (v) - {\rm sn}^2 (v) \, {\rm dn}^2 (v) = {\rm cn}^2 (v) \, {\rm dn}^2 (v)$$
on account of the (true) Jacobian identities $k^2 {\rm sn}^2 + {\rm dn}^2 = 1$ and ${\rm sn}^2 + {\rm cn}^2 = 1$. 
\end{proof} 

\medbreak 

Here, we may extract meromorphic square-roots: as the functions $c, \, {\rm cn}$ and ${\rm dn}$ all take the value $1$ at the origin, we deduce that $c$ is elliptic, being given by  
$$c (u) = {\rm cn} \, [(1 + k^2)^{-\tfrac{1}{2}} \, u] \, {\rm dn} \, [(1 + k^2)^{-\tfrac{1}{2}} \, u].$$ 
Less concretely: it may be checked that $\kappa^2 c^2 = d^2 - \lambda^2$; Theorem \ref{dvalues} gives $c^2$ quadruple poles and double zeros (at the lattice midpoints) so $c^2$ has meromorphic square-roots. 

\medbreak 

We have now recovered and elaborated upon almost all of the results in [2] whose statements do not involve theta functions, at the same time offering fresh proofs that are free of theta functions. The few such results that remain may safely be left as exercises along the same lines: for example, the functions 
$$(\phi\,')^2 = 2 \, \frac{d^2}{d + 1}$$
and 
$$(\psi\,')^2 = 2 \, \frac{d^2 - \lambda^2}{d + 1}$$
are plainly elliptic; the task of deciding whether they have meromorphic square-roots is such an exercise. As an exercise beyond these, locate precisely the zeros $z^{\pm}$ of Theorem \ref{dzero}; in this connexion, note that $\wp(z^{\pm}) = \tfrac{1}{2} \kappa^2 - \tfrac{1}{3}$ and $\wp(z^{\pm} + \ii \, \omega') = \tfrac{1}{6}$.  

\medbreak 

\section*{Appendix} 

It will perhaps be helpful to indicate the logical origins of the elliptic functions that are constructed from hypergeometric functions by Shen in [1] and [2]. 

\medbreak 

Classically, the elliptic functions of Jacobi arise by the inversion of elliptic integrals. Thus: when 
$$u = \int_0^{\phi} (1 - \kappa^2 \sin^2 \theta)^{-1/2} {\rm d} \theta$$
and the relation $\phi \mapsto u$ is inverted to $u \mapsto \phi$ near $0$ as a fixed point, $\phi$ is called the amplitude ${\rm am} \, u$ of $u$; the functions ${\rm sn}$ and ${\rm cn}$ are defined by 
$${\rm sn} \, u = \sin {\rm am} \, u = \sin \phi$$ 
$${\rm cn} \, u = \cos {\rm am} \, u = \cos \phi$$ 
while ${\rm dn}$ is defined either by 
$$(1) \; \; \; {\rm dn} \, u = \frac{{\rm d} \phi}{{\rm d} u}$$
or by 
$$(2) \; \; {\rm dn} \, u = (1 - \kappa^2 \sin^2 \phi)^{1/2}$$
as these two alternatives are equivalent in this classical setting. The integral that expresses $u$ in terms of $\phi$ may also be written 
$$u = \int_0^{\sin \phi} (1 - \kappa^2 t^2)^{-1/2} \frac{{\rm d} t}{\sqrt{1 - t^2}}$$
via the substitution $t = \sin \theta$ and by 
$$u = \int_0^{\sin \phi} F(\tfrac{1}{2}, \tfrac{1}{2}; \tfrac{1}{2} ; \kappa^2 t^2)  \frac{{\rm d} t}{\sqrt{1 - t^2}}$$
in view of the hypergeometric identity 
$$F(\tfrac{1}{2}, \tfrac{1}{2}; \tfrac{1}{2} ; z) = (1 - z)^{-1/2}.$$ 

\bigbreak 

In [1] the (classical) hypergeometric function $F(\tfrac{1}{2}, \tfrac{1}{2}; \tfrac{1}{2} ; \bullet)$ is replaced by $F(\tfrac{1}{3}, \tfrac{2}{3}; \tfrac{1}{2} ; \bullet)$: there, inversion of 
$$u = \int_0^{\sin \phi} F(\tfrac{1}{3}, \tfrac{2}{3}; \tfrac{1}{2} ; \kappa^2 t^2)  \frac{{\rm d} t}{\sqrt{1 - t^2}}$$
leads to functions ${\rm sn}_3$ and ${\rm cn}_3$ defined by
$${\rm sn}_3 \, u = \sin \phi$$ 
$${\rm cn}_3 \, u = \cos \phi$$ 
and to ${\rm dn}_3$ defined by
$${\rm dn}_3 \, u = \frac{{\rm d} \phi}{{\rm d} u}$$
in line with the classical choice (1) above. 

\medbreak 

In [2] the (classical) hypergeometric function $F(\tfrac{1}{2}, \tfrac{1}{2}; \tfrac{1}{2} ; \bullet)$ is replaced by $F(\tfrac{1}{4}, \tfrac{3}{4}; \tfrac{1}{2} ; \bullet)$: there, inversion of 
$$u = \int_0^{\sin \phi} F(\tfrac{1}{4}, \tfrac{3}{4}; \tfrac{1}{2} ; \kappa^2 t^2)  \frac{{\rm d} t}{\sqrt{1 - t^2}}$$
leads to functions ${\rm sn}_2$ and ${\rm cn}_2$ defined by
$${\rm sn}_2 \, u = \sin \phi$$ 
$${\rm cn}_2 \, u = \cos \phi$$ 
and to ${\rm dn}_2$ defined by
$${\rm dn}_2 \, u = (1 - \kappa^2 \sin^2 \phi)^{1/2}$$
in line with the classical choice (2) above. 

\bigbreak

\begin{center} 
{\small R}{\footnotesize EFERENCES}
\end{center} 
\medbreak 

[0] P.L. Robinson, {\it Elliptic functions from $F(\frac{1}{3}, \frac{2}{3} ; \frac{1}{2} ; \bullet)$}, arXiv 1907.09938 (2019). 

\medbreak 

[1] Li-Chien Shen, {\it On the theory of elliptic functions based on $_2F_1(\frac{1}{3}, \frac{2}{3} ; \frac{1}{2} ; z)$}, Transactions of the American Mathematical Society {\bf 357}  (2004) 2043-2058. 

\medbreak 

[2] Li-Chien Shen, {\it On a theory of elliptic functions based on the incomplete integral of the hypergeometric function $_2 F_1 (\frac{1}{4}, \frac{3}{4} ; \frac{1}{2} ; z)$}, Ramanujan Journal {\bf 34} (2014) 209-225. 

\medbreak

\end{document}